\documentclass[a4paper, reqno]{amsart}
\usepackage[utf8]{inputenc}
\usepackage[T1]{fontenc}
\usepackage{amssymb,amsthm,amsmath}
\usepackage[british]{babel}
\usepackage{calrsfs}
\usepackage{hyperref}
\usepackage{mathbbol}
\usepackage{mathabx}
\usepackage{fancyvrb}
\usepackage[all]{xy}
\usepackage[table]{xcolor}

\theoremstyle{plain}
\newtheorem{theorem}[subsection]{Theorem}
\newtheorem{proposition}[subsection]{Proposition}
\newtheorem{corollary}[subsection]{Corollary}
\theoremstyle{remark}
\newtheorem{example}[subsection]{Example}
\newtheorem{remark}[subsection]{Remark}

\newcommand{\bak}{\ensuremath{\backslash}}

\newcommand{\Mon}{\ensuremath{\mathsf{Mon}}}
\newcommand{\Grp}{\ensuremath{\mathsf{Grp}}}
\newcommand{\LPM}{\ensuremath{\mathsf{LPM}}}
\newcommand{\N}{\ensuremath{\mathbb{N}}}
\newcommand{\Z}{\ensuremath{\mathbb{Z}}}

\DeclareMathOperator{\Ima}{Im}
\DeclareMathOperator{\F}{F}

\DeclareMathOperator{\nf}{nf}

\newdir{ >>}{{}*!/8pt/@{|}*!/3.5pt/:(1,-.2)@^{>}*!/3.5pt/:(1,+.2)@_{>}}
\newdir{ |>}{{}*!/-3.5pt/@{|}*!/-8pt/:(1,-.2)@^{>}*!/-8pt/:(1,+.2)@_{>}}
\newdir{ >}{{}*!/-8pt/@{>}}
\newdir{>}{{}*:(1,-.2)@^{>}*:(1,+.2)@_{>}}
\newdir{<}{{}*:(1,+.2)@^{<}*:(1,-.2)@_{<}}

\def\pullback{
	\ar@{-}[]+R+<6pt,-1pt>;[]+RD+<6pt,-6pt>%
	\ar@{-}[]+D+<1pt,-6pt>;[]+RD+<6pt,-6pt>}

\def\ophalfsplitpullback{%
	\ar@{-}[]+R+<6pt,-1pt>;[]+RD+<6pt,-6pt>%
	\ar@{-}[]+D+<.5ex,-6pt>;[]+RD+<6pt,-6pt>}

\begin{document}

\title[Comparison between weakly protomodular and protomodular objects]{A comparison between weakly protomodular\\ and protomodular objects in unital categories}

\author[X.~Garc\'{\i}a-Mart\'{\i}nez]{Xabier Garc\'{\i}a-Mart\'{\i}nez}
\author[A.~Montoli]{Andrea Montoli}
\author[D.~Rodelo]{Diana Rodelo}
\author[T.~Van~der Linden]{Tim Van~der Linden}

\email{xabier.garcia.martinez@uvigo.gal}
\email{andrea.montoli@unimi.it}
\email{drodelo@ualg.pt}
\email{tim.vanderlinden@uclouvain.be}

\address[Xabier Garc\'{\i}a-Mart\'{\i}nez]{CITMAga \& Universidade de Vigo, Departamento de Ma\-tem\'{a}ticas, Esc.\ Sup.\ de Enx.\ Inform\'atica, Campus de Ourense, E--32004 Ourense, Spain}
\address[Andrea Montoli]{Dipartimento di Matematica ``Federigo Enriques'', Universit\`{a} degli Studi di Milano, Via Saldini 50, 20133 Milano, Italy}
\address[Diana Rodelo]{Departamento de Matem\'{a}tica, Universidade do Algarve, 8005-139 Faro, Portugal; Center for Research and Development in Mathematics and Applications (CIDMA), Department of Mathematics, University of Aveiro, 3810-193 Aveiro, Portugal; and CMUC, Departamento de Matem\'{a}tica, Universidade de Coimbra, 3000-143 Coimbra, Portugal}
\address[Tim Van~der Linden]{Institut de
	Recherche en Math\'ematique et Physique, Universit\'e catholique
	de Louvain, che\-min du cyclotron~2 bte~L7.01.02, B--1348
	Louvain-la-Neuve, Belgium; and Mathematics and Data Science, Vrije Universiteit Brussel, Pleinlaan 2, B--1050 Brussel, Belgium}

\thanks{The first author is supported by Ministerio de Ciencia e Innovaci\'on (Spain), with grant number PID2024-155502NB-I00 and by the Xunta de Galicia through the Competitive Reference Groups (GRC), ED431C 2023/31. The second author is member of the Gruppo Nazionale per le Strutture Algebriche, Geometriche e le loro Applicazioni (GNSAGA) dell'Istituto Nazionale di Alta Matematica ``Francesco Severi''. The third author acknowledges partial financial support by CIDMA under the Portuguese Foundation for Science and Technology (FCT, https://ror.org/00snfqn58) Multi-Annual Financing Program for R\&D Units, grants UID/4106/2025 and UID/PRR/4106/2025, and \emph{Centro de Matemática da Universidade de Coimbra} (CMUC), funded by the Portuguese Government through FCT/MCTES, UIDB/00324/2020. The fourth author is a Senior Research Associate of the Fonds de la Recherche Scientifique--FNRS. This work was supported by the Shota Rustaveli National Science Foundation of Georgia (SRNSFG), through grant FR-24-9660.}

\begin{abstract}
	We compare the concepts of \emph{protomodular} and \emph{weakly protomodular} objects within the context of unital categories. Our analysis demonstrates that these two notions are generally distinct. To establish this, we introduce \emph{left pseudocancellative unital magmas} and characterise weakly protomodular objects within the variety of algebras they constitute. Subsequently, we present an example of a weakly protomodular object that is not protomodular in this category.
\end{abstract}


\maketitle	

\section{Introduction}
Many of the intrinsic properties of non-abelian algebraic structures, such as groups, rings, Lie algebras and several others, have been successfully described in categorical terms thanks to the notion of \emph{semi-abelian category}~\cite{Janelidze-Marki-Tholen}. Semi-abelian categories play a central role in many recent developments of Categorical Algebra, such as non-abelian homology and cohomology (see, for instance,~\cite{Bourn-Janelidze:Torsors, BR2012, EGVdL, HVdL, RVdL2}).

A fundamental ingredient in the definition of a semi-abelian category is \emph{protomodularity}~\cite{Bourn1991}. A pointed category is protomodular if and only if the \emph{Split Short Five Lemma} holds in it. There are other equivalent ways of defining protomodularity. One of them makes use of the notion of \emph{jointly extremal-epimorphic} pair of morphisms: two morphisms \(f\colon A \to C\) and \(g\colon B \to C\) with the same codomain are jointly extremal-epimorphic if, whenever they factor through a common monomorphism \(m \colon M \rightarrowtail C\), then \(m\) is an isomorphism. In concrete algebraic contexts, this means that the object \(C\) is generated by the images of \(f\) and \(g\). A \emph{point}, namely a pair \((f, s)\) in which \(f\) is a split epimorphism and \(s\) is a chosen section of \(f\), is said to be \emph{strong} if the section \(s\) and the kernel \(k\) of \(f\) are jointly extremal-epimorphic (here we are supposing that the category is finitely complete and pointed, although the definition of a strong point can be formulated in non-pointed categories as well). The point \((f, s)\) is called \emph{stably strong} if every point obtained as a pullback, along any morphism, of \((f, s)\) is strong. A pointed finitely complete category is protomodular if and only if every point in it is strong or, equivalently, every point is stably strong (see e.g.~\cite{Borceux-Bourn} for a proof of this fact).

In order to find a purely categorical characterisation of groups amongst monoids, in~\cite{MRVdL-TCOGAM} the authors considered a local version of protomodularity, introducing the notion of \emph{protomodular object} inside a (not necessarily protomodular) category. An object \(X\) in a pointed finitely complete category is such if every point over~\(X\), i.e., every point \((f, s)\) in which \(X\) is the codomain of the split epimorphism \(f\), is stably strong. Obviously, a category is protomodular if and only if every object in it is protomodular. The class of protomodular objects in a category satisfies strong algebraic properties, as shown in~\cite{MRVdL-TCOGAM}. In the category of monoids, the protomodular objects are precisely the groups; similarly, in the category of semirings, the protomodular objects are the rings. Note that the protomodularity of an object depends on the category where it is considered. For instance, any group is a protomodular object in the category \(\Grp\) of groups and in the category \(\Mon\) of monoids, but it is not protomodular in the category of unital magmas unless it is trivial---see Theorem~\ref{Theorem Unital Magmas} for an explicit proof of this fact.

Later, in~\cite{GM-ACS}, the author considered a weaker notion, the concept of \emph{weakly protomodular object}: an object \(X\) is such when every point over \(X\) is strong. He showed that in \(\Mon\) the weakly protomodular objects are precisely the groups, so protomodular and weakly protomodular objects coincide in \(\Mon\). Similarly, one can show that the same happens for semirings: the protomodular and weakly protomodular objects are the rings.

It is easy to see that every protomodular object is a weakly protomodular object. The converse is false, since there are weakly protomodular objects that are not protomodular. For example, in every pointed finitely complete category the zero object is weakly protomodular, while, as observed in~\cite[Proposition 7.12]{MRVdL-TCOGAM}, the condition that the zero object is protomodular is equivalent to the category being unital in the sense of~\cite{B96}. Recall that a pointed finitely complete category is a \emph{unital category} if, for every pair of objects \((X, Y)\) the canonical morphisms \({\langle 1_X, 0 \rangle \colon X \to X \times Y}\) and  \(\langle 0, 1_Y \rangle \colon Y \to X \times Y\) are jointly extremal-epimorphic. So, for example, in the category of pointed sets the singleton is a weakly protomodular object but not a protomodular object, since this category is not unital.

In all previously studied examples of unital categories, protomodular and weakly protomodular objects have been found to coincide. This includes cases such as monoids and semirings, as mentioned above, but also cocommutative bialgebras over an algebraically closed field, where (weakly) protomodular objects are precisely the cocommutative Hopf algebras, as demonstrated in~\cite{GM-VdL1}. This consistency raises the question of whether protomodular and weakly protomodular objects always coincide in unital categories.

In this paper we give a negative answer to this question. To this aim, we introduce the algebraic structure of \emph{left pseudocancellative unital magmas} (briefly LPM). Left pseudocancellative unital magmas and morphisms between them form a unital category, denoted by \(\LPM\). We characterise the weakly protomodular objects in \(\LPM\) (Theorem \ref{T:wprot}) and we show that all the subobjects of a protomodular object must be weakly protomodular (Theorem \ref{T:prot}). Then we exhibit a weakly protomodular object in \(\LPM\) with a subobject which is not weakly protomodular, showing that it is not a protomodular object. In conclusion, the notions of protomodular and weakly protomodular objects do not coincide in all unital categories.

\section{The case of left pseudocancellative unital magmas}

We start by proving a claim made in the Introduction:

\begin{theorem}\label{Theorem Unital Magmas}
	The variety of unital magmas does not admit non-trivial (weakly) protomodular objects.
\end{theorem}
\begin{proof}
	Let \(M\) be a unital magma with unit denoted by \(e\). Suppose \(m\in M\). Let \(F\) be the free unital magma on one element \(x\), and \(\overline{m}\colon F\to M\) the morphism sending \(x\) to \(m\). Consider the split epimorphism \({\lgroup \overline{m} \;\; 1_M \rgroup\colon F+M \to M}\) and the induced point determined by the coproduct inclusion \(\iota_M\colon M\to F+M\). If this point is strong, then \(x\) in \(F+M\) can be written as a product of elements in the kernel \(K\) of \(\lgroup \overline{m} \;\; 1_M \rgroup\) with elements in \(M\). Yet, the only ways to write \(x\) as a product in \(F+M\) are \(1x\), \(x1\), \(1(x1)\), \(1(1x)\), \((x1)1\), etc., where \(1\) denotes the unit of \(F+M\). Necessarily then, \(x\in K\), which implies that \(m=\lgroup \overline{m} \;\; 1_M \rgroup(x)=e\). Hence \(M\) is trivial.
\end{proof}

We recall that a \emph{left quasigroup} is a set \(Q\) equipped with a binary operation \(* \colon Q\times Q \to Q\), such that for all \(y\in Q\), the \emph{left multiplication map}
\[
	M_y \colon Q \to Q\colon x \mapsto y*x
\]
is bijective. Left quasigroups form a variety of universal algebras. In fact, a left quasigroup can be equivalently defined as a set \(Q\) with two binary operations: a multiplication denoted by \(*\) and a left division denoted by \(\bak \,\), satisfying the following identities:
\begin{align}
	y & = x * (x \bak y) \label{identities_leftloop1} \\
	y & = x \bak (x * y) \label{identities_leftloop2}
\end{align}

A left quasigroup with an \emph{identity element}, i.e., an element \(e \in Q\) such that
\begin{equation}
	x = e*x = x*e, \label{identities_leftloop3}
\end{equation}
for all \(x \in Q\), it is called a \emph{left loop}.

If \(Q\) is a left loop, then for each \(x \in Q\), we have that
\begin{equation}\label{x bak x=e}
	x \bak x = x \bak (x * e) \stackrel{\eqref{identities_leftloop2}}{=} e.
\end{equation}
It is known that the category of left loops is semi-abelian~\cite{BoJa}, thus it is a pro\-to\-mo\-du\-lar category. Consequently, all of its objects are (weakly) protomodular.

In this work we consider a weakening of the concept of left loop. We will say that a set \(X\) with two binary operations \(*\) and \(\bak\) and a nullary operation \(e\) is a \emph{left pseudocancellative unital magma} (\emph{LPM} for short) if it satisfies identities~\eqref{identities_leftloop1} and~\eqref{identities_leftloop3}. We have not found any reference of this structure in the literature; the name is inspired by \emph{left cancellative magmas}, which are sets with two operations \(*\) and \(\bak\) satisfying identity~\eqref{identities_leftloop2}. Note that, in general, the identity \(x\bak x=e\) need not hold in an LPM: see Proposition~\ref{if x div x=e then protomodular}. We observe that axiom~\eqref{identities_leftloop1} can be equivalently reformulated by saying that, for an algebra \(X\) and any element \(y \in X\), we have \(M_y\circ D_y=1_X\): the left multiplication map
\[
	M_y \colon X \to X\colon x \mapsto y*x
\]
is a left inverse for the map
\[
	D_y \colon X \to X\colon x \mapsto y \bak x.
\]
Hence we get:

\begin{proposition}\label{Prop with 3 pps}
	Let \(X\) be an LPM. Then:
	\begin{enumerate}
		\item[(i)] all maps \(M_y\) are surjective and all maps \(D_y\) are injective;
		\item[(ii)] \(D_e=1_X\);
		\item[(iii)] if \(x \bak y = e\), then \(x = y\).
	\end{enumerate}
\end{proposition}

\begin{proof}
	(i) is obvious. (ii) holds because \(D_e(x)=e\bak x \stackrel{\eqref{identities_leftloop3}}{=} e*(e\bak x) \stackrel{\eqref{identities_leftloop1}}{=} x\). (iii) If \(x \bak y = e\), then
	\[
		x \stackrel{\eqref{identities_leftloop3}}{=} x * e = x * (x\bak y) \stackrel{\eqref{identities_leftloop1}}{=} y.\qedhere
	\]
\end{proof}

\begin{corollary}
	Any finite LPM is a left loop.
\end{corollary}
\begin{proof}
	A surjection from a finite set to itself is automatically injective.
\end{proof}

\begin{proposition}
	Any set \(X\) with a chosen element \(e\) and a binary operation~\(\bak\) such that all maps \(D_y \colon X \to X\colon x \mapsto y \bak x\) are injective and conditions (ii) and (iii) of Proposition~\ref{Prop with 3 pps} hold, admits an LPM structure by defining the \(*\) operation via
	\[
		y*x=
		\begin{cases}
			D_y^{-1}(x) & \text{if \(x \in \Ima D_y\)}     \\
			y           & \text{if \(x \notin \Ima D_y\).}
		\end{cases}
	\]
\end{proposition}
\begin{proof}
	Note that the operation \(*\) is well defined, since it is built by specifying a left inverse of the injective maps \(D_y\). Axiom \eqref{identities_leftloop1} is then obvious.
	
	To prove \eqref{identities_leftloop3}: \(e*x=e*D_e(x)=e*(e\bak x) \stackrel{\eqref{identities_leftloop1}}{=} x\). On the other hand, if \(e\in \Ima D_x\), then \(e=x\bak a\), for some \(a\in X\). By the assumption that condition (iii) of Proposition~\ref{Prop with 3 pps} holds, \(x=a\). Then, \(x*e=x*(x\bak x) \stackrel{\eqref{identities_leftloop1}}{=} x\). If \(e\notin \Ima D_x\), then \(x*e=x\) by definition.
\end{proof}

Let us denote the category of all LPM, and morphisms between them, by \(\LPM\). Since LPM form a variety of universal algebras, \(\LPM\) is a complete and cocomplete category. Moreover, since LPM have a unique constant in their signature, \(\LPM\) is a pointed variety. It was shown in~\cite{Borceux-Bourn} that a pointed variety of universal algebras is a unital category if and only if it contains, among its operations, those of unital magmas. By equation~\eqref{identities_leftloop3}, this allows us to conclude that \(\LPM\) is a unital category. \(\LPM\) is not a protomodular category, as we will see below by proving that not every object in it is protomodular (Example~\ref{E:nwp}). In order to characterise the weakly protomodular objects in the category \(\LPM\) (Theorem~\ref{T:wprot}), we first show that the elements in a free LPM admit unique normal forms, by means of a simple term rewriting argument.

We orient the axioms of left pseudocancellative unital magmas through the following set of rewrite rules:
\[
	R = \bigl\{
	x*(x\backslash y) \to y,\;
	e * x \to x,\;
	x * e \to x,\;
	e \backslash y \to y
	\bigr\}.
\]
The last rule is a consequence of the axioms, since substituting \(e\) for \(x\) in equation\eqref{identities_leftloop1} yields \(y = e*(e\backslash y)\), which reduces to \(e\backslash y\) by the unit law~\eqref{identities_leftloop3}.

\begin{proposition}\label{rewriting}
	The rewriting system \(R\) is terminating and confluent. Hence every element of
	a free LPM admits a unique normal form.
\end{proposition}

\begin{proof}
	\emph{Termination.}
	Define a measure \(\mu(t) = (\#_\ast(t), \#_{\backslash}(t))\in \mathbb{N}\times \mathbb{N}\), where \(\#_\ast(t)\) (resp.\ \(\#_{\backslash}(t)\)) is the number of occurrences of \(*\) (resp.\ \(\backslash\)) in a chosen representation of an element \(t\) of a free LPM. We order \(\mathbb{N}\times \mathbb{N}\) lexicographically. Each rule strictly decreases this measure: the first three remove one occurrence of \(*\), the first and last remove one occurrence of \(\backslash\). Thus no infinite reduction exists.
	
	\emph{Confluence.}
	By Newman's Lemma, it suffices to check local confluence. The left-hand sides are
	\[
		\begin{aligned}
			L_1(x,y) & = x*(x\backslash y), \\
			L_2(x)   & = e*x,               \\
			L_3(x)   & = x*e,               \\
			L_4(y)   & = e\backslash y.
		\end{aligned}
	\]
	The only non-trivial overlap occurs when \(x=e\) in \(L_1\), giving the term
	\(e*(e\backslash y)\). We have
	\[
		e*(e\backslash y) \xrightarrow{L_1} y,
		\qquad
		e*(e\backslash y) \xrightarrow{L_2} e\backslash y \xrightarrow{L_4} y,
	\]
	so the critical pair is joinable. All other overlaps are trivial. Hence the system
	is locally confluent.
\end{proof}

\begin{theorem}\label{T:wprot}
	Let \(X\) be an LPM. The following statements are equivalent:
	\begin{enumerate}
		\item[(i)] \(X\) is a weakly protomodular object in \(\LPM\);
		\item[(ii)] for any \(x \in X\), there exist \(x_1\), \dots, \(x_n \in X\) such that
		      \begin{equation} \label{condition for weak protomodularity}
			      x_1\bak( x_2 \bak ( \cdots \bak (x_n \bak x) \cdots )) = e.
		      \end{equation}
	\end{enumerate}
\end{theorem}

\begin{proof}
	(ii) \(\Rightarrow\) (i) We consider any split extension
	\[
		\xymatrix{0 \ar[r] & K \ar@{ |>->}[r]^-{k} & Y \ar@<.5ex>[r]^-{f} & X \ar@{ >->}@<.5ex>[l]^-{s} \ar[r]& 0}
	\]
	and we take any \(y \in Y\). Let \(x = f(y)\) and let \(x_i\) be the elements whose existence is guaranteed by (ii).
	Then by identity \eqref{identities_leftloop1},
	\[
		y = s(x_n) * \bigl(	\cdots *\bigl(s(x_1)*\bigl(s(x_1)\bak(s(x_2) \bak ( \cdots \bak (s(x_n) \bak y) \cdots ))\bigr) \bigr)\cdots \bigr),
	\]
	where clearly \(s(x_1)\bak(s(x_2) \bak ( \cdots \bak (s(x_n) \bak y) \cdots ))\) belongs to the kernel \(K\) of \(f\), thanks to \eqref{condition for weak protomodularity}. This proves that any element \(y\in Y\) can be written as a product of elements in~\(K\) and elements of the image of \(s\). Hence the point \((f,s)\) is strong and, consequently, \(X\) is a weakly protomodular object.
	
	(i) \(\Rightarrow\) (ii) We choose any \(x \in X\) and we consider the split extension
	\[
		\xymatrix{0 \ar[r] & K \ar@{ |>->}[r]^-{k} & \F(z) + X \ar@<.5ex>[r]^-{\lgroup \overline{x} \;\; 1_X \rgroup} & X \ar@{ >->}@<.5ex>[l]^-{\iota_{X}} \ar[r]& 0,}
	\]
	where \(\F(z)\) denotes the free LPM on one generator \(z\not\in X\), \(\overline{x}\colon \F(z)\to X\) denotes the universal morphism taking \(z\) to \(x\),  \(\iota_X\colon X\to \F(z)+X\) is the coproduct inclusion, and  \(\lgroup \overline{x} \;\; 1_X \rgroup\) is induced by the universal property of the coproduct.
	
	Since \(X\) is weakly protomodular, \(z \in \F(z) + X\) can be written as a word consisting of elements of \(X\) and elements of the kernel \(K\). Specifically, there are natural numbers \(r\) and \(\ell\) and elements \(k_1\), \dots, \(k_r\) of \(K\), \(x_1\), \dots, \(x_\ell\) of \(X\), and \(w\) of the free algebra \(\F(k_1,\dots,k_r,x_1,\dots,x_\ell)\) on those elements, such that the canonical morphism
	\[
		\varphi\colon \F(k_1,\dots,k_r,x_1,\dots,x_\ell)\to \F(z)+X
	\]
	maps \(w\) to \(z\).
	
	By Proposition~\ref{rewriting}, in the free algebra \(F(\{z\}\cup X)\), each element has a unique irreducible representative (normal form) with respect to \(R\).  Concretely, for each \(t\) there is a unique irreducible \(\nf(t)\) such that \(t \xrightarrow{*}_R \nf(t)\), and \(\nf(t)\) is the normal form of the element represented by \(t\).
	
	Let us, in particular, write
	\[
		\psi\colon \F(k_1,\dots,k_r,x_1,\dots,x_\ell)\to \F(\{z\}\cup X)
	\]
	for the canonical morphism sending each \(k_i\) to any term in \(\F(\{z\}\cup X)\) in normal form which evaluates to \(k_i\) in \(\F(z)+X\). The element \(z\in \F(\{z\}\cup X)\) has the normal form \(\mathrm{nf}(z)=z\) (the single-letter term \(z\) is irreducible), and the equality \(\varphi(w)=z\) in the coproduct algebra means that the term \(\psi(w)\in \F(\{z\}\cup X)\) reduces to the irreducible term \(z\), so that \(\psi(w)\xrightarrow{*}_R z\).
	
	Consider the finite reduction sequence \(\psi(w) \xrightarrow{*}_R z\). If, in the word \(\psi(w)\), we replace all but one of the instances of \(z\) by the letter \(x\), we still find a term which reduces to \(z\), by precisely the same steps that allow us to reduce \(\psi(w)\) to \(z\). Thus we find an element \(k\) of \(K\), elements \(x_1\), \dots, \(x_\ell\) of \(X\), and \(v\) in the free algebra \(\F(k,x_1,\dots,x_\ell)\) such that the canonical morphism
	\[
		\phi\colon \F(k,x_1,\dots,x_\ell)\to \F(z)+X
	\]
	maps \(v\) to \(z\).
	
	Because the final normal form is the single letter \(z\), the last rewrite step in this sequence must be an instance of the rule
	\[
		u*(u\bak y)\to y
	\]
	(with \(y\) instantiated to \(z\)); indeed, a step back in the reduction cannot be a consequence of the rules \(L_3\), \(L_2\) or \(L_4\), because the \(k\in K\) which contains the letter \(z\) was chosen to be in normal form. In other words, immediately before the final step the term has the shape
	\[
		u*(u\bak z)
	\]
	for some sub-term \(u\) built from the letters in \(X\cup\{z\}\). After that last rewrite we obtain \(z\).
	
	We now iteratively analyse the term \(u\). If \(u\) still contains the distinguished letter~\(z\), then the reduction of \(\phi(v)\) to \(z\) must have produced that occurrence of~\(z\) in \(u\) at a previous step; tracing back the finite reduction we obtain an explicit finite sequence of nested preimages showing that \(u\) is obtained by iterating the same pattern \(t \mapsto r*(r\bak t)\) a finite number of times. Concretely, by unwinding the final portion of the reduction sequence we find a finite depth \(m\) and elements \(y_1\), \dots, \(y_m\) from \(X\) (each arising as a root symbol of a sub-term that is built only from letters of \(X\)) such that the sub-term \(u\bak z\) has the syntactic shape
	\[
		y_1\bak\bigl(y_2\bak(\cdots \bak (y_m\bak z)\cdots)\bigr).
	\]
	Thus the rightmost part of the term immediately before the final step is exactly of the form
	\[
		u*(u\bak z)\quad\text{with}\quad u\bak z
		= y_1\bak\bigl(y_2\bak(\cdots\bak(y_m\bak z)\cdots)\bigr),
	\]
	and each \(y_i\) belongs to \(X\).
	
	We recall that \(u\bak z\) is the rightmost part of a word \(w\) in \(\F(z) + X\) made by words in \(X\) and in the kernel \(K\) of \(\lgroup \overline{x} \;\; 1_X \rgroup\) and we observe that its rightmost part must be a word in \(K\), since \(z\) does not belong to \(X\). So the rightmost part of the word \(u\bak z\) is a word \(s \in K\) of the form
	\[
		s \;\coloneq \; x_1\bak\bigl(x_2\bak(\cdots\bak(x_n\bak z)\cdots)\bigr),
	\]
	where \(x_i = y_{i+(j-1)}\) for some \(1\leq j \leq m\) and \(n=m-j+1\).
	
	This means that, applying \(\lgroup \overline{x} \;\; 1_X \rgroup\) to \(s\) (so, replacing \(z\) with \(x\) in the word \(s\)) we obtain the equality
	\[
		x_1\bak\bigl(x_2\bak(\cdots\bak(x_n\bak x)\cdots)\bigr) = e
	\]
	in \(X\)---which concludes the proof.
\end{proof}

\begin{corollary}
	Any left loop is a weakly protomodular object in \(\LPM\).
\end{corollary}

\begin{proof}
	By identity \eqref{x bak x=e} and Theorem~\ref{T:wprot}.
\end{proof}

\begin{example}\label{E:nwp}
	Let \(\N\) be the set of natural numbers with the operations
	\[
		x\bak y  =
		\left\{\begin{array}{lll}
			y   & \text{if \(x = 0\)} & (a)  \\
			y+1 & \text{if \(x > 0\)} & (b)
		\end{array}\right.
		\qquad
		x * y    =
		\left\{\begin{array}{lll}
			x   & \text{if \(y = 0\)}                & (a')   \\
			y   & \text{if \(x = 0\)}                & (b')   \\
			y-1 & \text{if \(x > 0\) and \(y > 0\) } & (c').
		\end{array}\right.
	\]
	We check that \((\N, *, \bak ,0)\) forms an LPM.
	\begin{enumerate}
		\item \(x*(x\bak y)=y\) holds for all possible cases:
		      \begin{itemize}
			      \item if \(x=0\), then \(0*(0\bak y)\stackrel{(a)}{=} 0*y \stackrel{(b')}{=} y\);
			      \item  if \(x>0\), then \(x*(x\bak y)\stackrel{(b)}{=} x*(y+1) \stackrel{(c')}{=} y\).
		      \end{itemize}
		      \setcounter{enumi}{2}
		\item \(x*0=x=0*x\) holds by \((a')\) and \((b')\).
	\end{enumerate}
	This LPM is not weakly protomodular, since any \(x > 0\) does not satisfy condition (ii) in Theorem~\ref{T:wprot}. Therefore, the category \(\LPM\) is not protomodular.
\end{example}

\begin{theorem}\label{T:prot}
	In the category \(\LPM\), the subalgebras of a protomodular object are weakly protomodular.
\end{theorem}

\begin{proof}
	Let \(Y\) be a protomodular object in \(\LPM\) and \(X\) a subalgebra of \(Y\). Given an element \(x \in X\), we show that condition (ii) in Theorem~\ref{T:wprot} holds. Let us consider the pullback diagram
	\[
		\xymatrix@C=4em{
		K \ar@{ |>->}[d]_-{} \ar@{=}[r] & K \ar@{ |>->}[d]^{} \\
		P \ophalfsplitpullback \ar@<.5ex>[d]^(.6){\pi_X} \ar@{ >->}[r]^-{\pi_{\F(z) + Y}} & \F(z) + Y \ar@<.5ex>[d]^-{\lgroup \overline{x} \;\; 1_Y \rgroup} \\
		X \ar@{ >->}@<.5ex>[u]^-{\langle 1_{X}, \iota_{Y} \circ i \rangle} \ar@{{ >}->}[r]_-i & Y, \ar@{ >->}@<.5ex>[u]^-{\iota_{Y}}
		}
	\]
	where \(i\) is the inclusion, \(\F(z)\) is the free LPM on one generator \(z\), \(\overline{x}\) denotes the universal morphism sending \(z\) to \(x\), \(\iota_Y\) is the coproduct inclusion, \(\lgroup \overline{x} \;\; 1_Y \rgroup\) is induced by the universal property of the coproduct, and \(K\) is the kernel of the vertical split epimorphisms. Note that the pullback \(P\) consists of those words of \(F(z) + Y\) such that, when \(z\) is replaced by \(x\) and the operations are computed, the resulting word belongs to \(X\). Since we are assuming that \(Y\) is a protomodular object, the left vertical point is strong. Moreover, the element $z$ belongs to $P$. Therefore, using exactly the same argument as in the proof of Theorem~\ref{T:wprot}, we end up with a word of the form
	\[
		x_1 \bak \bigl(x_2 \bak( \cdots \bak (x_{n}  \bak z) \cdots )\bigr)
	\]
	which belongs to the kernel \(K\) of \(\lgroup \overline{x} \;\; 1_Y \rgroup\). Hence, applying \(\lgroup \overline{x} \;\; 1_Y \rgroup\) to such word (which means replacing \(z\) with \(x\)), we see that \(x\) satisfies condition (ii) in Theorem~\ref{T:wprot}.
\end{proof}

\begin{proposition} \label{if x div x=e then protomodular}
	Let \(X\) be an LPM such that \(x \bak x = e\) for all \(x \in X\). Then \(X\) is a protomodular object in \(\LPM\).
\end{proposition}

\begin{proof}
	Note that \(X\) is obviously a weakly protomodular object, by Theorem~\ref{T:wprot}. Let us prove that \(X\) is, moreover, a protomodular object by showing that any point \((f,s)\) over \(X\) is stably strong. Consider the following pullback of \((f,s)\) along an arbitrary morphism \(g\):
	\[
		\xymatrix@C=4em{
		K \ar@{ |>->}[d]_-{} \ar@{=}[r] & K \ar@{ |>->}[d]^{} \\
		P \ophalfsplitpullback \ar@<.5ex>[d]^(.6){\pi_Y} \ar[r]^-{\pi_{Z}} & Z \ar@<.5ex>[d]^-{f} \\
		Y \ar@{ >->}@<.5ex>[u]^-{\langle 1_{Y}, s \circ g \rangle} \ar[r]_-g & X. \ar@{ >->}@<.5ex>[u]^-{s}
		}
	\]
	Here \(K\) is the kernel of the vertical split epimorphisms.
	A pair \((y, z)\in Y\times Z\) belongs to \(P\) if an only if \(g(y) = f(z)\). Then, thanks to our assumption, which implies that
	\[
		f(sg(y) \bak z)=g(y)\bak f(z)=e=g(e)\text{,}
	\]
	we can decompose \((y, z)\) in \(P\) as
	\[
		(y, z) = (y, sg(y)) * (e, sg(y) \bak z ),
	\]
	where the first pair belongs to the image of \(\langle 1_{Y}, s \circ g \rangle\) and the second one belongs to the kernel \(K\) of \(\pi_Y\). This proves that the point \((\pi_Y,\langle 1_{Y}, s \circ g \rangle)\) is strong; thus \((f,s)\) is a stably strong point.
\end{proof}

\begin{corollary}
	Any left loop is a protomodular object in \(\LPM\).
\end{corollary}

\begin{proof}
	By \eqref{x bak x=e} and Proposition~\ref{if x div x=e then protomodular}.
\end{proof}

Now we can show that there exists a weakly protomodular object in \(\LPM\) which is not protomodular. To do that, we consider the following example.

\begin{example}\label{E: wproto, not proto}
	Let \(\Z\) be the set of integers with the operations
	\[
		x\bak y  =
		\left\{\begin{array}{lll}
			y     & \text{if \(x = 0\)}                         & (a)  \\
			y+1   & \text{if \(x > 0\) and \(y \geq 0\)}        & (b)  \\
			y     & \text{if \(x > 0 > y\)}                     & (c)  \\
			-2y-1 & \text{if \(x < 0 \leq y\)}                  & (d)  \\
			2y    & \text{if \(x\), \(y < 0\) and \(x \neq y\)} & (e)  \\
			0     & \text{if \(x=y < 0\)}                       & (f)
		\end{array}\right.
	\]
	and
	\[
		x * y    =
		\left\{\begin{array}{lll}
			x              & \text{if \(y = 0\)}            & (a')  \\
			y              & \text{if \(x = 0\)}            & (b')  \\
			y-1            & \text{if \(x\), \(y > 0\)}     & (c')  \\
			y              & \text{if \(x > 0 > y\)}        & (d')  \\
			\frac{-y-1}{2} & \text{if \(x < 0\), \(y\) odd} & (e')  \\
			\frac{y}{2}    & \text{if \(x < 0\neq y\) even} & (f')
		\end{array}\right.
	\]
	on \(\Z\). We check that \((\Z,*,\bak,0)\) forms an LPM.
	\begin{enumerate}
		\item \(x*(x\bak y)=y\) holds for all the possible cases:
		      \begin{itemize}
			      \item if \(x=0\), then \(0*(0\bak y)\stackrel{(a)}{=} 0*y \stackrel{(b')}{=} y\);
			      \item if \(x>0\), \(y\ge 0\), then \(x*(x\bak y)\stackrel{(b)}{=} x*(y+1) \stackrel{(c')}{=} y\);
			      \item if \(x>0>y\), then \(x*(x\bak y)\stackrel{(c)}{=} x*y \stackrel{(d')}{=} y\);
			      \item if \(x<0\le y\), then \(x*(x\bak y)\stackrel{(d)}{=} x*(-2y-1) \stackrel{(e')}{=} \dfrac{-(-2y-1)-1}{2}=y\);
			      \item if \(x,y<0\), \(x\neq y\), then \(x*(x\bak y)\stackrel{(e)}{=} x*(2y) \stackrel{(f')}{=} \dfrac{2y}{2}=y\);
			      \item if \(x=y<0\), then \(x*(x\bak y)\stackrel{(f)}{=} x*0 \stackrel{(a')}{=} x=y\).
		      \end{itemize}
		      \setcounter{enumi}{2}
		\item \(x*0=x=0*x\) holds by \((a')\) and \((b')\).
	\end{enumerate}
	Moreover, \((\Z,*,\bak,0)\) is weakly protomodular, since it satisfies condition (ii) in Theorem~\ref{T:wprot}:
	\[
		0 =
		\begin{cases}
			x \bak x                & \text{if \(x \leq 0\)} \\
			(-2x-1)\bak (-1 \bak x) & \text{if \(x > 0\).}
		\end{cases}
	\]
	It is straightforward to check that the operations defined above for \(\Z\), restricted to natural numbers, are precisely those given in Example~\ref{E:nwp}. The non--weakly protomodular LPM of Example~\ref{E:nwp} is then a subalgebra of this LPM. Therefore, by Theorem~\ref{T:prot} we have found an LPM which is weakly protomodular but not protomodular.
\end{example}

\begin{remark} We can generalise Example~\ref{E: wproto, not proto} as follows. Consider \(\Z\) with the operations
	\[
		x\bak y  =
		\left\{\begin{array}{lll}
			y     & \text{if \(x = 0\)}                         & (a)  \\
			y+i   & \text{if \(x > 0\) and \(y \geq 0\)}        & (b)  \\
			y     & \text{if \(x > 0 > y\)}                     & (c)  \\
			-ky-j & \text{if \(x < 0 \leq y\)}                  & (d)  \\
			ky    & \text{if \(x\), \(y < 0\) and \(x \neq y\)} & (e)  \\
			0     & \text{if \(x=y < 0\)}                       & (f)
		\end{array}\right.
	\]
	and
	\[
		x * y    =
		\left\{\begin{array}{lll}
			x                  & \text{if \(y = 0\)}                    & (a')                      \\
			y                  & \text{if \(x = 0\)}                    & (b')                      \\
			y-i                & \text{if \(x\), \(y > 0\)}             & (c')                      \\
			y                  & \text{if \(x > 0 > y\)}                & (d')                      \\
			\frac{-y-(k-1)}{k} & \text{if \(x < 0\), \(y\equiv_k 1\)}   & (e_1')   \vspace{5pt}     \\
			\frac{-y-(k-2)}{k} & \text{if \(x < 0\), \(y\equiv_k 2\)}   & (e_2')                    \\
			\quad\vdots                                                                             \\
			\frac{-y-j}{k}     & \text{if \(x < 0\), \(y\equiv_k k-j\)} & (e_{k-j}')                \\
			\quad\vdots                                                                             \\
			\frac{-y-1}{k}     & \text{if \(x < 0\), \(y\equiv_k k-1\)} & (e_{k-1}')   \vspace{5pt} \\
			\frac{y}{k}        & \text{if \(x < 0\neq y, y\equiv_k 0\)} & (f'),
			
		\end{array}\right.
	\]
	on \(\Z\), for fixed integers \(i\geq 1\), \(k\ge 2\), \(j\in\{1, \dots, k-1\}\). The proof that this is an LPM is similar to that of Example~\ref{E: wproto, not proto}. To show that \(x*(x\bak y)=y\) in the case \(x<0\le y\), we calculate:
	\[
		x*(x\bak y)\stackrel{(d)}{=} x*(-ky-j) \stackrel{(e_{k-j}')}{=} \dfrac{-(-ky-j)-j}{k}=y.
	\]
	It is also a weakly protomodular object, since condition (ii) in Theorem~\ref{T:wprot} is satisfied:
	\[
		0 =
		\begin{cases}
			x \bak x                & \text{if \(x \leq 0\)} \\
			(-kx-j)\bak (-1 \bak x) & \text{if \(x > 0\).}
		\end{cases}
	\]
\end{remark}

The next example shows that there exists an LPM \(X\) satisfying the identity \(x\bak x=e\)---thus, by Proposition~\ref{if x div x=e then protomodular}, a protomodular object in \(\LPM\)---that is not a left loop.
\begin{example}\label{E:protnonloop}
	Let \(\N\) be the set of natural numbers with the operations
	\begin{align*}
		x\bak y  =
		\begin{cases}
			y   & \text{if \(x = 0\)}                  \\
			0   & \text{if \(x = y\)}                  \\
			y+1 & \text{if \(x > 0\) and \(x \neq y\)}
		\end{cases}
		\qquad\qquad
		x * y    =
		\begin{cases}
			x   & \text{if \(y = 0\)}         \\
			y   & \text{if \(x = 0\)}         \\
			y-1 & \text{if \(x\), \(y > 0\).}
		\end{cases}
	\end{align*}
	Using similar arguments to those in Example~\ref{E:nwp}, it is easy to check that \((\N,*,\bak,0)\) is an LPM. Thanks to Proposition~\ref{if x div x=e then protomodular}, it is a protomodular object in \(\LPM\): for any element \(x \in \mathbb{N}\), we have that \(x \bak x = 0\). Nevertheless, it is not a left loop, since
	\[
		1 \bak (1 * 2) = 1 \bak 1 = 0 \neq 2.
	\]
\end{example}

This, together with Example~\ref{E: wproto, not proto}, proves that the classes of left loops, weakly protomodular objects in \(\LPM\) and protomodular objects in \(\LPM\) are strictly included in each other:
\[
	\{\text{left loops}\} \;\subsetneq\; \{\text{protomodular objects}\} \;\subsetneq\; \{\text{weakly protomodular objects}\}.
\]

\section*{Acknowledgements}
The authors wish to express their gratitude to Manuel Ladra and the University of Santiago de Compostela for their kind hospitality. Many thanks also to the referee for their careful reading of the first version of this article.

\end{document}